\documentclass[12pt]{amsart}

\usepackage{amssymb}
\usepackage{bm}
\usepackage{enumerate}
\usepackage{graphicx}
\usepackage{psfrag}
\usepackage{color}
\usepackage{url}

\usepackage{algpseudocode}

\usepackage{xy}
\input xy
\xyoption{all}

\newcommand{\excise}[1]{}

\newtheorem{thm}{Theorem}[section]
\newtheorem{lemma}[thm]{Lemma}

\newtheorem{cor}[thm]{Corollary}
\newtheorem{prop}[thm]{Proposition}

\newtheorem{prob}[thm]{Problem}

\theoremstyle{definition}

\newtheorem{example}[thm]{Example}
\newtheorem{remark}[thm]{Remark}
\newtheorem{defn}[thm]{Definition}

\numberwithin{equation}{section}

\newcommand{\ring}[1]{\ensuremath{\mathbb{#1}}}

\renewcommand\>{\rangle}
\newcommand\<{\langle}

\newcommand\NN{\ring{N}}

\newcommand\ZZ{\ring{Z}}

\DeclareMathOperator\supp{supp} 

\DeclareMathOperator\Betti{Betti} 
\DeclareMathOperator\Cong{Cong} 

\begin{document}

\mbox{}
\title{Minimal presentations of shifted numerical monoids}

\author[Conway]{Rebecca Conaway}
\address{Monmouth University\\West Long Branch, NJ 07764}
\email{s0969183@monmouth.edu}

\author[Gotti]{Felix Gotti}
\address{Mathematics Department\\UC Berkeley\\Berkeley, CA 94720}
\email{felixgotti@berkeley.edu}

\author[Horton]{Jesse Horton}
\address{University of Arkansas\\Fayetteville, AR 72701}
\email{jphorton@uark.edu}

\author[O'Neill]{Christopher O'Neill}
\address{Mathematics Department\\Texas A\&M University\\College Station, TX 77843}
\email{coneill@math.tamu.edu}

\author[Pelayo]{Roberto Pelayo}
\address{Mathematics Department\\University of Hawai`i at Hilo\\Hilo, HI 96720}
\email{robertop@hawaii.edu}

\author[Williams]{Mesa Williams}
\address{Lee University\\Cleveland, TN 37311}
\email{mprach00@leeu.edu}

\author[Wissman]{Brian Wissman}
\address{Mathematics Department\\University of Hawai`i at Hilo\\Hilo, HI 96720}
\email{wissman@hawaii.edu}

\date{\today}

\begin{abstract}
A numerical monoid is an additive submonoid of the non-negative integers.  Given a numerical monoid $S$, consider the family of ``shifted'' monoids $M_n$ obtained by adding $n$ to each generator of $S$.  In this paper, we examine minimal relations among the generators of $M_n$ when $n$ is sufficiently large, culminating in a description that is periodic in the shift parameter $n$.  We explore several applications to computation, combinatorial commutative algebra, and factorization theory.  
\end{abstract}

\maketitle


\section{Introduction} \label{sec:intro}


A minimal presentation of a numerical monoid $M$ (that is, an additive submonoid of the natural numbers) encapsulates the minimal relations among generators of $M$.  Such minimal presentations arise in the study of toric ideals, where they correspond to minimal generating sets for kernels of monomial maps \cite{clo}, and algebraic statistics, where they correspond to Markov bases \cite{algmarkov}.  Additionally, many arithmetic invariants of interest in combinatorial commutative algebra and factorization theory can be easily recovered (both theoretically and computationally) from a minimal presentation, making them a particularly useful tool in computational algebra~\cite{compoverview,numerical}.  

In this paper, we examine families of numerical monoids obtained by ``shifting'' a chosen generating set.  In particular, given positive integers $r_1 < \cdots < r_k$, consider numerical monoids of the form 
$$M_n = \<n, n+r_1, \ldots, n + r_k\>,$$
indexed by a shift parameter $n$.  Our main result is Theorem~\ref{t:minpresbij}, which describes how minimal presentations of $M_n$ vary with large $n$.  More specifically, we give an explicit bijection between the minimal presentations of $M_n$ and those of $M_{n + r_k}$ when $n > r_k^2$.  

Following Theorem~\ref{t:minpresbij}, we characterize the behavior of several arithmetic invariants determined by minimal presentations (e.g.\ Betti numbers and catenary degree), resulting in periodic or periodic-linear descriptions in each case.  
Some of these characterizations are new, while others strengthen existing results in the literature \cite{deltashiftgen,vu14}.  Our approach unifies these results (old and new) as consequences of a deeper structural phenomenon that occurs among the minimal relations of $M_n$ for large $n$, and improves each lower bound on $n$ that was previously given; see Remark~\ref{r:vucompare} for a thorough discussion of the benefits of our approach and resulting improvements.  

One of the primary consequences of Theorem~\ref{t:minpresbij} lies in the realm of computation.  While minimal presentations (and many of the arithmetic invariants they determine) are generally more difficult to compute for monoids with large generators, our results give a way to more efficiently perform these computations in some cases by instead computing a minimal presentation for a numerical monoid with smaller generators in the same shifted family.  We discuss the specifics in Remark~\ref{r:minprescomputation}, including a forthcoming implementation in the popular \texttt{GAP} package \texttt{numericalsgps} \cite{numericalsgpsgap}.  

\section{Background}
\label{sec:background}

In this section, we provide the necessary definitions related to the factorization theory of numerical monoids.  In what follows, let $\NN$ denote the set of non-negative integers.  

\begin{defn}\label{d:numericalmonoid}
A \emph{numerical monoid} $M$ is an additive submonoid of $\NN$.  When we write $M = \<m_1, \ldots, m_t\>$, we assume $m_1 < \cdots < m_t$, and the chosen generators $m_1, \ldots, m_t$ are called \emph{irreducible} elements or \emph{atoms}.  We say $M$ is \emph{primitive} if $\gcd(m_1, \ldots, m_t) = 1$.  
\end{defn}

\begin{defn}\label{d:factorization}
Fix a numerical monoid $M = \<m_1, \ldots, m_t\>$ and $a \in M$.  A \emph{factorization} of $a$ is an expression 
$$a = z_1m_1 + \cdots + z_tm_t$$
of $a$ as a sum of irreducible elements of $M$, which we often represent with the tuple $z = (z_1, \ldots, z_t) \in \NN^t$.  The \emph{length} of a factorization $z$ of $a$ is the total number
$$|z| = z_1 + \cdots + z_t$$
of irreducible elements appearing in $z$, and the \emph{support} of $z$ is the set
$$\supp(z) = \{m_i : z_i > 0\}$$
of distinct irreducible elements appearing in $z$.  
\end{defn}


\begin{defn}\label{d:facthomo}
Fix a numerical monoid $M = \<m_1, \ldots, m_t\>$ and $a \in M$.  The \emph{factorization homomorphism} of $M$ is the map $\pi:\NN^t \to M$ given by
$$\pi(z_1, \ldots, z_t) = z_1 m_1 + \cdots + z_t m_t.$$
The \emph{set of factorizations} of $a$ is the set
$$\mathsf Z_M(a) = \pi^{-1}(a) = \{z \in \NN^t : \pi(z) = a\} \subset \NN^t.$$
When there can be no confusion, we often omit the subscript and simply write $\mathsf Z(a)$.  
\end{defn}

We conclude this section with Theorem~\ref{t:minfactlen}, which appeared as \cite[Theorem~4.3]{elastsets} for minimally generated, primitive numerical monoids.  The statement below follows immediately from the proof of the original statement given in \cite{elastsets}.  

\begin{thm}\label{t:minfactlen}
Fix $m_1 < \cdots < m_t$, and suppose $M = \<m_1, \ldots, m_t\>$ is not necessarily primitive or minimally generated by $m_1, \ldots, m_t$.  The function $\mathsf m:M \to \NN$ sending each $a \in M$ to its smallest factorization length satisfies
$$\mathsf m(a + m_t) = \mathsf m(a) + 1$$
for all $a > m_{t-1}m_t$.  
\end{thm}

\subsection*{Notation}
Through the remainder of this paper, fix $r_1 < \cdots < r_k$ and $n \in \NN$, and let 
$$S = \<r_1, \ldots, r_k\> \qquad \text{ and } \qquad M_n = \<n, n + r_1, \ldots, n + r_k\>$$
denote additive submonoids of $\NN$.  Unless otherwise stated, we assume $M_n$ is primitive and minimally generated as written, but we do \textbf{not} make either assumption for $S$.  Note that choosing $n$ as the first generator of $M_n$ ensures that every numerical monoid falls into exactly one shifted family.

\section{Sufficiently shifted numerical monoids}
\label{sec:sufficientshift}

In this section, we give the Theorem~\ref{t:mesalemma} and Corollary~\ref{c:bettilengthset}, which identify the core obstruction to Theorem~\ref{t:minpresbij} for small $n$ (in the sense of Remark~\ref{r:boundorigin}).  This result comes in the form of a description of the factorizations of Betti elements (Definition~\ref{d:betti}), whose factorizations encapsulate the minimal relations among atoms.  

\begin{defn}\label{d:betti}
Fix a numerical monoid $M$ and $a \in M$.  The \emph{factorization graph} of $a$, denoted $\nabla\!_a$, has vertex set $\mathsf Z(a)$, and two vertices $z, z' \in \mathsf Z(a)$ are connected by an edge whenever they have at least one irreducible in common.  We say $a$ is a \emph{Betti element} of $M$ if its factorization graph $\nabla\!_a$ is disconnected, and write $\Betti(M)$ for the set of Betti elements of $M$.  
\end{defn}

\begin{example}\label{e:betti}
The Betti elements of $M = \<6,9,20\>$ are 18 and 60, since
$$\begin{array}{rcl}
\mathsf Z(18) &=& \{(3,0,0), (0,2,0)\} \quad \text{ and } \\
\mathsf Z(60) &=& \{(10,0,0), (7,2,0), (4,4,0), (1,6,0), (0,0,3)\}
\end{array}$$
both yield disconnected factorization graphs.  Here, the factorizations of 18 represent the minimal relation between 6 and 9, namely that in any factorization of an element $a \in M$, one can replace three copies of 6 with two copies of 9 to yield a new factorization of $a$.  Similarly, 60 is the first element that can be factored using all three irreducibles, and thus gives the minimal ways to exchange copies of 20 for copies of 6 and 9.

In contrast, the element $126 \in M$ is not a Betti element of $M$, even though $(1,0,6), (0,14,0) \in \mathsf Z(126)$ have no irreducibles in common.  Indeed, this relation can be obtained by twice exchanging three 20's for ten 6's, yielding $(21,0,0)$, and then repeatedly exchanging three 6's for two 9's until $(0,14,0)$ is obtained.  This is represented by a path through the factorization graph $\nabla\!_{126}$ connecting $(1,0,6)$ to $(0,14,0)$ that passes through 9 vertices, including $(21,0,0)$.  
\end{example}


Before stating and proving Theorem~\ref{t:mesalemma} and Corollary~\ref{c:bettilengthset}, we prove Lemma~\ref{l:boundsrc}, which identifies the locations in the proof of Theorem~\ref{t:mesalemma} that require $n$ to be sufficiently large.   Note that Theorem~\ref{t:mesalemma} is the source of the bound given in nearly every ``eventual behavior'' result in this paper; see Remark~\ref{r:boundorigin} for more detail.  

\begin{lemma}\label{l:boundsrc}
Fix $a \in S$ and $s = (s_1, \ldots, s_k) \in \mathsf Z(a)$.  
\begin{enumerate}[(a)]
\item 
\label{l:boundsrc_rk}
If $|s| \ge r_k$ and $s_k = 0$, then there is a shorter factorization $s' \in \mathsf Z(a)$ with $s_k' > 0$.  

\item 
\label{l:boundsrc_bound}
If $a > r_{k-1}r_k$ and $s$ has minimum factorization length, then $s_k > 0$.   

\item 
\label{l:boundsrc_lenbound}
If $a > r_k^2$, then $|s| \ge r_k$.  

\end{enumerate}
\end{lemma}

\begin{proof}
\cite[Lemma~4.1]{elastsets} and Theorem~\ref{t:minfactlen}.  
\end{proof}

\begin{thm}\label{t:mesalemma}
Suppose $n > r_k^2$, and let $z$ and $z'$ be factorizations of a Betti element $\beta \in M_n$ in different connected components of $\nabla\!_\beta$.  If $|z| > |z'|$, then $z_0 > 0$ and $z'_k > 0$.
\end{thm}

\begin{proof}
We begin by observing that
$$\beta - |z|n = z_0n + \sum_{i = 1}^k z_i(n + r_i) - |z|n = \sum_{i = 1}^k z_ir_i,$$
which yields an explicit bijection between the factorizations of $\beta \in M_n$ of length $\ell$ and the factorizations of $\beta - \ell n \in S$ of length at most $\ell$.  Let $s = (z_1, \ldots, z_k) \in \mathsf Z_S(\beta - |z|n)$ and $s' = (z_1', \ldots, z_k') \in \mathsf Z_S(\beta - |z'|n)$ denote the factorizations in $S$ corresponding to $z$ and $z'$, respectively.  
Notice that since $|z| > |z'|$, we have 
$$\beta - |z'|n \ge n + \beta - |z|n \ge n > r_k^2,$$
so $|s'| \ge r_k$ by Lemma~\ref{l:boundsrc}\eqref{l:boundsrc_lenbound}.  

Next, we claim some factorization in the same connected component of $\nabla\!_\beta$ as $z'$ has positive last component.  Certainly if $z_k' > 0$ the claim is proved.  Otherwise, since $\beta - |z'|n \ge r_{k-1}r_k$, applying Lemma~\ref{l:boundsrc}\eqref{l:boundsrc_bound} produces a factorization $s'' \in \mathsf Z_S(\beta - |z'|n)$ with $s_k'' > 0$ obtained from $s'$ by replacing all but one atom with a minimum length factorization.  The corresponding factorization $z'' = (|z'| - |s''|,s_1'', \ldots, s_k'') \in \mathsf Z(\beta)$ of $\beta$ under the above bijection is connected to $z'$ in $\nabla\!_\beta$ and has $z_0'' > 0$ and $z_k'' > 0$.  

Now, since $z$ and $z'$ lie in different connected components of $\nabla\!_\beta$, the above claim implies $z_k = 0$.  This means $|s| \le r_k$, since otherwise Lemma~\ref{l:boundsrc}\eqref{l:boundsrc_rk} would produce a factorization connected to $z$ in $\nabla\!_\beta$ with positive last coordinate.  As such,  
$$|z| > |z'| \ge |s'| \ge r_k \ge |s|,$$
which yields $z_0 = |z| - |s| > 0$.  Lastly, we conclude $z_k' > 0$, as otherwise the factorization $z''$ constructed above would be connected to both $z$ and $z'$ in $\nabla\!_\beta$ since it has positive first coordinate.  
\end{proof}

\begin{cor}\label{c:bettilengthset}
Suppose $n > r_k^2$ and that $M_n$ is primitive, and let $d = \gcd(r_1, \ldots, r_k)$.  Any two factorizations $z, z' \in \mathsf Z(\beta)$ of a Betti element $\beta \in \Betti(M_n)$ lying in different connected components of $\nabla\!_\beta$ satisfy $\big| |z| - |z'| \big| \in \{0,d\}$.
\end{cor}

\begin{proof}
By \cite[Proposition~2.9]{delta}, $|z| - |z'| \in d\ZZ$, so suppose by way of contradiction that $|z| - |z'| \ge 2d$.  Since $\beta - |z'|n - n \in S$, there exists a factorization $z'' \in \mathsf Z(\beta)$ with $|z''| = |z'| + d$.  By Theorem~\ref{t:mesalemma}, both $z_0''$ and $z_k''$ must be positive, meaning $z$ and $z'$ are both connected to $z''$ in $\nabla\!_\beta$.  
\end{proof}

\section{Minimal presentations}
\label{sec:minpres}

Let $\pi_n:\NN^{k+1} \to M_n$ denote the factorization homomorphism of $M_n$, that is,
$$\pi_n(z) = z_0n + \sum_{i = 1}^k z_i(n + r_i),$$
and let $\ker \pi_n$ denote the equivalence relation on $\NN^{k+1}$ given by $(z,z') \in \ker\pi_n$ whenever $\pi_n(z) = \pi_n(z')$, (that is, when $z$ and $z'$ are factorizations for the same element in $M_n$).  The equivalence relation $\ker \pi_n$ is a \emph{congruence} since it is also closed under \emph{translation}, that is, $(z + u, z' + u) \in \ker \pi_n$ whenever $(z,z') \in \ker\pi_n$ and $u \in \NN^{k+1}$.  

\begin{defn}\label{d:minpres}
Fix a numerical monoid $M = \<m_1, \ldots, m_t\>$ and let $\pi:\NN^t \to M$ denote the factorization homomorphism of $M$.  A \emph{presentation} for $M$ is a set of relations $\rho \subset \ker \pi$ such that $\ker \pi$ is the unique minimal (w.r.t.\ containment) congruence on $\NN^t$ containing $\rho$.  Equivalently, this is true if between any two factorizations $(z, z') \in \ker \pi$, there exists a \emph{chain} $a_0, a_1, \ldots, a_r$ with $a_0 = z$, $a_r = z'$, and 
$$(a_{i-1},a_i) = (b_i,b_i') + (u_i,u_i) \in \ker \pi$$
for some $(b_i, b_i') \in \rho$ and $u_i \in \NN^t$ for each $i \le r$.  We say $\rho$ is \emph{minimal} if it is minimal with respect to containment among all presentations of $M$.  
\end{defn}

Minimal presentations are one of the fundamental tools with which to study the factorization structure of finitely generated monoids.  Each minimal presentation of a monoid $M$ can be viewed as a particular choice of minimal relations that are sufficient for relating any two factorizations of the elements of $M$.  For a more thorough introduction, we refer the reader to \cite[Chapter~9]{fingenmon} and \cite[Chapter~7]{numerical}.  

\begin{example}\label{e:minpres}
The minimal presentations of $M = \<6,9,20\>$ from Example~\ref{e:betti} are 
$$\begin{array}{ll}
\{((3,0,0), (0,2,0)), ((10,0,0), (0,0,3))\}, & \{((3,0,0), (0,2,0)), ((7,2,0), (0,0,3))\}, \\
\{((3,0,0), (0,2,0)), ((\phantom{0}4,4,0), (0,0,3))\}, & \{((3,0,0), (0,2,0)), ((1,6,0), (0,0,3))\}, 
\end{array}$$
each of which has exactly one relation for each Betti element.  As per the discussion in Example~\ref{e:betti}, each minimal presentation provides enough relations among the minimal generators of $M$ to relate any two factorizations of elements of $M$.
\end{example}

\begin{example}\label{e:minpresmap}
Let $S = \<6,9,20\>$, and consider the following minimal presentations.  
\smaller
$$
\begin{array}{
r@{\,\,\,}
l@{}r@{\,\,}r@{\,\,}r@{\,\,}r@{\,}r@{\,\,}r@{\,\,}r@{\,}r@{}l@{\,\,}
l@{}r@{\,\,}r@{\,\,}r@{\,\,}r@{\,}r@{\,\,}r@{\,\,}r@{\,}r@{}l@{\,\,}
l@{}r@{\,\,}r@{\,\,}r@{\,\,}r@{\,}r@{\,\,}r@{\,\,}r@{\,}r@{}l@{\,\,}
}
M_{450}:
& (( &  0, & 0, & 8, & 0), & (3, & 2, & 0, & 3  & )),
& (( &  0, & 1, & 6, & 0), & (4, & 0, & 0, & 3  & )),
& (( &  0, & 3, & 0, & 0), & (1, & 0, & 2, & 0  & )),
\\
& (( & 20, & 5, & 0, & 0), & (0, & 0, & 0, & 24 & )),
& (( & 25, & 1, & 0, & 0), & (0, & 0, & 4, & 21 & )),
& (( & 26, & 0, & 0, & 0), & (0, & 2, & 2, & 21 & ))
\\[0.5em]

M_{470}:
& (( &  0, & 0, & 8, & 0), & (3, & 2, & 0, & 3  & )),
& (( &  0, & 1, & 6, & 0), & (4, & 0, & 0, & 3  & )),
& (( &  0, & 3, & 0, & 0), & (1, & 0, & 2, & 0  & )),
\\
& (( & 21, & 5, & 0, & 0), & (0, & 0, & 0, & 25 & )),
& (( & 26, & 1, & 0, & 0), & (0, & 0, & 4, & 22 & )),
& (( & 27, & 0, & 0, & 0), & (0, & 2, & 2, & 22 & ))
\\[0.5em]

M_{490}:
& (( &  0, & 0, & 8, & 0), & (3, & 2, & 0, & 3  & )),
& (( &  0, & 1, & 6, & 0), & (4, & 0, & 0, & 3  & )),
& (( &  0, & 3, & 0, & 0), & (1, & 0, & 2, & 0  & )),
\\
& (( & 22, & 5, & 0, & 0), & (0, & 0, & 0, & 26 & )),
& (( & 27, & 1, & 0, & 0), & (0, & 0, & 4, & 23 & )),
& (( & 28, & 0, & 0, & 0), & (0, & 2, & 2, & 23 & ))
\end{array}
$$
\normalsize
Each first-row relation $(z,z')$ satisfies $|z| = |z'|$, and each second-row relation $(z,z')$ satisfies $|z| = |z'| + 1$, $z_0 > 0$ and $z_3' > 0$.  Theorem~\ref{t:mesalemma} ensures that every relation above satisfies one of these two criteria.  

Theorem~\ref{t:minpresbij} characterizes the relationship between successive minimal presentations.  In particular, the same equal-length relations appear in all three given minimal presentations, and each remaining relation for $M_{n+20}$ is obtained from a relation for $M_n$ by adding $(e_0, e_k)$.  
\end{example}

Proposition~\ref{p:mapwelldefined} defines the map $\Phi_n$ used to construct the bijection between minimal presentations in Theorem~\ref{t:minpresbij}, and Proposition~\ref{p:mapproperties} gives several key properties of $\Phi_n$.  In particular, it is shown that $\Phi_n$ preserves symmetric and translation closure, and preserves monotone chain connectivity (Definition~\ref{d:monotonechain}).  

\begin{prop}\label{p:mapwelldefined}
The map $\Phi_n \colon \ker \pi_n \to \ker \pi_{n+r_k}$ given by
$$\Phi_n(z,z')
= \left\{\begin{array}{lll}
(z + \ell e_0, z' + \ell e_k) & \text{if } & |z| > |z'| \\
(z + \ell e_k, z' + \ell e_0) & \text{if } & |z| < |z'| \\
(z,z')                        & \text{if } & |z| = |z'|
\end{array}\right.
$$
for $(z,z') \in \ker \pi_n$ and  $\ell = \big| |z| - |z'| \big|$ is well defined.
\end{prop}

\begin{proof}
Fix $(z,z') \in \ker \pi_n$ with $z = (z_0, \dots, z_k)$ and $z' = (z_0', \dots, z_k')$.  By symmetry, we can assume that $\ell = |z| - |z'| \ge 0$.  Now, we simply use $\pi_n(z) = \pi_n(z')$ to verify that 
$$\begin{array}{r@{}c@{}l}
\displaystyle \pi_{n+r_k}(z+de_0)
&{}={}& \displaystyle
(|z| + \ell)(n + r_k) + \sum_{i = 1}^k z_ir_i
= (|z'| + 2\ell)(n + r_k) - \ell n + \sum_{i=1}^k z_i'r_i
\\ &{}={}& \displaystyle
(|z'| + \ell)(n + r_k) + \ell r_k + \sum_{i=1}^k z_i'r_i
= \pi_{n + r_k}(z' + \ell e_k),
\end{array}$$
as desired.  
\end{proof}

\begin{defn}\label{d:monotonechain}
A chain $z = a_0, a_1, \ldots, a_r = z'$ between factorizations $z, z'$ in a numerical monoid is \emph{monotone} if $|a_0|, |a_1|, \ldots, |a_r|$ is a monotone sequence.  
\end{defn}

\begin{prop}\label{p:mapproperties}
Fix $n$, $\rho \subset \ker \pi_n$, and $(z, z') \in \rho$, and let $(w, w') = \Phi_n(z, z')$.  
\begin{enumerate}[(a)]
\item
\label{p:mapproperties_injective}
The map $\Phi_n$ is injective.  

\item 
\label{p:mapproperties_lendiffs}
The map $\Phi_n$ preserves length differences: $|z| - |z'| = |w| - |w'|$.  

\item 
\label{p:mapproperties_closures}
The map $\Phi_n$ preserves the reflexive, symmetric, and translation closure operations: if $\rho$ is reflexive, symmetric, or closed under translation, then so is $\Phi_n(\rho)$.  

\item 
\label{p:mapproperties_monotonechains}
The map $\Phi_n$ preserves monotone chain connectivity: if $\rho$ is translation-closed and there exists a monotone $\rho$-chain from $z$ to $z'$, then there exists a monotone $\Phi_n(\rho)$-chain from $w$ to $w'$.  

\end{enumerate}
\end{prop}

\begin{proof}
It is easy to check that $|z| - |z'| = |w| - |w'|$, from which injectivity follows.  Both $\Phi(z,z) = (z,z)$ and $\Phi(z',z) = (w',w)$ follow directly from definitions as well.  Next, fixing $u \in \NN^{k+1}$ and assuming by symmetry that $\ell = |z| - |z'| \ge 0$, we have
$$\Phi_n(z + u, z' + u) = (z + u + \ell e_0, z' + u + \ell e_k) = (z + \ell e_0, z' + \ell e_k) + (u,u) = \Phi_n(z,z') + \Phi_n(u,u).$$
It remains to prove the final claim.  

Suppose $\rho$ is translation-closed and there is a monotone decreasing $\rho$-chain from $z$ to $z'$.  By induction on chain length, we can assume there is a single intermediate factorization $z''$.  Letting $\ell = |z| - |z''| \ge 0$ and $\ell' = |z''| - |z'| \ge 0$, we have
$$\Phi_n(z + \ell'e_0, z'' + \ell'e_0) = (z + (\ell + \ell')e_0, z'' + \ell'e_0 + \ell e_k)$$
and
$$\Phi_n(z'' + \ell e_k, z' + \ell e_k) = (z'' + \ell e_k + \ell 'e_0, z' + (\ell + \ell')e_k)$$
which form a monotone decreasing $\Phi_n(\rho)$-chain from $w$ to $w'$.  
\end{proof}

The main obstruction to Theorem~\ref{t:minpresbij} for arbitrary $n$ is that $\Phi_n$ needs only preserve connectivity by monotone chains.  Proposition~\ref{p:monotonechain} ensures that for $n$ sufficiently large, any pair of factorizations $(z,z') \in \ker \pi_n$ is connected by a monotone chain, and Example~\ref{e:nonmonotonechain} demonstrates why this can fail for small $n$.  

\begin{example}\label{e:nonmonotonechain}
Let $S = \<3,14\>$.  The element $1078 \in M_{74}$ has factorization set
$$\mathsf Z_{M_{74}}(1078) = \{(0, 14, 0), (11, 0, 3), (0, 6, 7)\}.$$
Notice that the only chains between $(0, 14, 0)$ and $(11, 0, 3)$ are a monotone chain directly between them and a non-monotone chain through $(0, 6, 7)$.  Since both relations in the non-monotone chain are translations, no minimal presentation of $M_{74}$ contains the relation $((0, 14, 0), (11, 0, 3))$.  On the other hand, we have
$$\mathsf Z_{M_{88}}(1274) = \{(0, 14, 0), (11, 0, 3)\},$$
so every minimal presentation of $M_{88}$ contains the relation $((0, 14, 0), (11, 0, 3))$.  
\end{example}

\begin{prop}\label{p:monotonechain}
Fix $n > r_k^2$ and a minimal presentation $\rho \subseteq \ker \pi_n$.  There exists a monotone $\rho$-chain between any $(z,z') \in \ker \pi_n$.
\end{prop}

\begin{proof}
Without loss of generality, assume $\gcd(z,z') = 0$.  By way of contradiction, assume there is no monotone $\rho$-chain from $z$ to $z'$.  Since $\Cong(\rho) = \ker \pi_n$, there exists a chain $z = a_0, a_1, \ldots, a_r = z'$ of factorizations such that for each $i < r$, we have 
$$(a_i,a_{i+1}) = (b_i,b_i') + (u_i,u_i), \qquad (b_i, b_i') \in \rho, u_i \in \NN^{k+1},$$
where $b_i$ and $b_i'$ occur in distinct connected components of the graph $\nabla\!_\beta$ of $\beta = \pi_n(b_i)$.  

By Corollary~\ref{c:bettilengthset}, we have $|a_i| - |a_{i-1}| \in \{-1,0,1\}$ for each $i \le r$.  As such, the sequence $|a_0|, |a_1|, \ldots, |a_r|$ of factorization lengths has non-sequential repeated values.  Without loss of generality, we can replace $z'$ with the factorization whose length is the first non-sequential repeated value in this sequence, and replace $z$ with the last factorization before $z'$ with $|z| = |z'|$.  As such, $|a_0| = |a_r| = |a_1| = |a_{r-1}|$, and $|a_i| = |a_{i-1}|$ whenever $1 < i < r$.  

First, suppose $|a_1| = |a_0| + 1$.  Applying Theorem~\ref{t:mesalemma} to the pairs $(b_1,b_1')$ and $(b_r,b_r')$, we see that $z_k > 0$ and $z_k' > 0$, which contradict the assumption that $\gcd(z,z') = 0$.  Likewise, if $|a_1| = |a_0| - 1$, then Theorem~\ref{t:mesalemma} implies that $z_0 > 0$ and $z_0' > 0$, which again contradict the assumption that $\gcd(z,z') = 0$.  This completes the proof.  
\end{proof}

Together, Propositions~\ref{p:mapproperties} and~\ref{p:monotonechain} yield Theorem~\ref{t:minpresbij}, the main result of this section.  

\begin{thm}\label{t:minpresbij}
For any $n > r_k^2$, the image of any minimal presentation $\rho$ of $M_n$ under the map $\Phi_n:\ker \pi_n \to \ker \pi_{n + r_k}$ is a minimal presentation of $M_{n + r_k}$.  In particular, $\Phi_n$~induces a one-to-one correspondence between the minimal presentations of $M_n$ and the minimal presentations of $M_{n+r_k}$.  
\end{thm}

\begin{proof}
We begin by showing that any minimal presentation $\rho \subset \ker \pi_n$ of $M_n$ satisfies
$$\Cong(\Phi_n(\rho)) = \ker \pi_{n + r_k},$$
that is, the image of $\rho$ under $\Phi_n$ is a presentation for $M_{n + r_k}$.  Fix $(w,w') \in \ker \pi_{n+r_k}$, and let $m = \pi_{n + r_k}(w)$.  By Proposition~\ref{p:monotonechain}, there exists a monotone chain from $w$ to $w'$, which we can assume is monotone decreasing by Proposition~\ref{p:mapproperties}\eqref{p:mapproperties_closures}.  We can also assume each step in this chain has the form $(b,b') + (u,u)$ for some $u \in \NN^{k+1}$ and $b, b' \in \mathsf Z(\beta)$ lying in different connected components of $\nabla\!_\beta$.  By Proposition~\ref{p:mapproperties}\eqref{p:mapproperties_closures}, it suffices to prove each $(b,b')$ lies in the image of $\Phi_n$, so it is enough to assume $w$ and $w'$ lie in different connected components of $\nabla\!_m$.  

First, if $|w| = |w'|$, then $\Phi_n(w,w') = (w,w')$ by Proposition~\ref{p:mapwelldefined}.  Otherwise, Corollary~\ref{c:bettilengthset} implies $|w| = |w'| + d$, where $d = \gcd(r_1, \ldots, r_k)$, and $w_0, w_k' \ge d$ follows from the proof of Theorem~\ref{t:mesalemma}.  This means $\Phi_n(w - de_0, w' - de_k) = (w,w')$, which proves $\Phi_n(\rho)$ generates $\ker \pi_{n + r_k}$.  

Now, by Propositions~\ref{p:mapproperties}\eqref{p:mapproperties_monotonechains} and~\ref{p:monotonechain}, factorizations $z,z' \in \mathsf Z_{M_n}(\beta)$ lie in distinct connected components of $\nabla\!_\beta$ if and only if $(w,w') = \Phi_n(z,z')$ lie in different connected components of the factorization graph of $\pi_{n + r_k}(w)$, so the image $\Phi_n(\rho)$ is indeed minimal as a presentation of $M_{n + r_k}$.  Additionally, the above argument implies that any minimal presentation $\rho'$ of $M_{n + r_k}$ is contained in the image of $\Phi_n$, so its preimage $\Phi_n^{-1}(\rho')$ is a minimal presentation for $M_n$.  This completes the proof.  
\end{proof}

\begin{remark}\label{r:boundorigin}
Resuming notation from Theorem~\ref{t:minpresbij}, the bound $n > r_k^2$ first appears in Theorem~\ref{t:mesalemma}, and this is the only result explicitly using the bound.  In particular, each subsequent result requiring $n > r_k^2$ (including Corollary~\ref{c:bettilengthset}, Proposition~\ref{p:monotonechain}, Theorem~\ref{t:minpresbij}, Corollary~\ref{c:eqlen}, and several results in Section~\ref{sec:invariants}) only uses this bound to (possibly indirectly) apply Theorem~\ref{t:mesalemma}.  As such, any improvement on the bound in Theorem~\ref{t:mesalemma} immediately improves Theorem~\ref{t:minpresbij}.  
\end{remark}

\begin{example}\label{e:minpressizes}
For fixed $S$ and $n$ sufficiently large, the size of a minimal presentation for $M_n$ need not be fixed within a given $r_k$-period.  For example, if $S = \<6,9,20\>$, the size of a minimal presentation of $M_n$ ranges from 4 (for $n = 420$) to 8 (for $n = 417$).  
\end{example}

We conclude the section with Corollary~\ref{c:eqlen}, which uses Proposition~\ref{p:monotonechain} to characterize the minimal relations for $M_n$ whose factorizations have equal length.  

\begin{cor}\label{c:eqlen}
Fix $n > r_k^2$ and a minimal presentation $\rho \subset \ker \pi_n$ for $M_n$.  Then
$$\tau = \{((z_1, \ldots, z_k), (z_1', \ldots, z_k')) : (z,z') \in \rho \text{ and } |z| = |z'|\} \subset \NN^k \times \NN^k$$
is a presentation for $S$.  
\end{cor}

\begin{proof}
Fix $s,s' \in \mathsf Z_S(s)$ with $|s| \le |s'|$.  Let $\pi:\NN^k \to S$ denote the factorization homomorphism of $S$, and let $z = (|s'| - |s|, s_1, \ldots, s_k)$ and $z' = (0, s_1', \ldots, s_k')$.  Then $(z,z') \in \ker \pi_n$, so by Proposition~\ref{p:monotonechain} there exists a monotone $\rho$-chain from $z$ to $z'$, but since $|z| = |z'|$, each factorization $a$ in the chain must have length $|z|$ as well.  As such, 
$$\pi_n(a) = |z|n + \sum_{i = 1}^k a_ir_i = |z|n + \pi(a_1, \ldots, a_k)$$
for each $a$ in the chain, thus producing a $\tau$-chain from $s$ to $s'$.  
\end{proof}

\begin{example}\label{e:eqlen}
Resuming notation from Corollary~\ref{c:eqlen}, the presentation $\tau$ for $S$ need not be minimal.  Indeed, if $S = \<6,9,20\>$, then the minimal presentation $\rho$ of $M_{450}$ given in Example~\ref{e:minpresmap} contains three equal-length relations, yielding the presentation
$$\tau = \{((0, 8, 0), (2, 0, 3)), ((1, 6, 0), (0, 0, 3)), ((3, 0, 0), (0, 2, 0))\}$$
for $S$.  The first of the above relations is redundant, as the latter two form a minimal presentation $\tau'$ for $S$.  However, the only $\tau'$-chain between $(0, 8, 0)$ and $(2, 0, 3)$ is non-monotone, which is why $\rho$ must also contain the relation $((0, 0, 8, 0), (3, 2, 0, 3))$.  
\end{example}

\section{Applications to factorization invariants}
\label{sec:invariants}

In this section, we explore several consequences of the results in Sections~\ref{sec:sufficientshift} and~\ref{sec:minpres}.  We begin with Remark~\ref{r:minprescomputation}, which discusses computational applications of Theorem~\ref{t:minpresbij}, and Corollary~\ref{c:bettiperiodic}, which improves a recent result from commutative algebra (see Remark~\ref{r:syzygies} for more on this connection).  Next, we characterize the behavior of several arithmetical invariants of non-unique factorization over $M_n$ for large $n$.  The survey article \cite{numericalsurvey} gives an overview of several of the invariants discussed here.  

\begin{remark}\label{r:minprescomputation}
Minimal presentations are used frequently in computer software package implementations, since many quantities of interest can then be quickly computed~\cite{compoverview}.  Additionally, minimal presentations have particular significance in commutative algebra; see Remark~\ref{r:syzygies}.  Most existing algorithms to compute a minimal presentation of a given numerical monoid use Gr\"obner basis techniques, which become computationally infeasible as the number and size of the generators grow large~\cite{clo}.  

Theorem~\ref{t:minpresbij} yields a method of reducing this complexity in certain cases.  In particular, if the generators of $M = \<n_1, \ldots, n_k\>$ satisfy $n_1 > (n_k - n_1)^2$, then a minimal presentation for $M$ can be computed by first computing a minimal presentation for $M' = \<n_1 - R, \ldots, n_k - R\>$, where $R$ is some appropriately chosen multiple of $n_k - n_1$, and then successively applying the map $\Phi_*$ from Proposition~\ref{p:mapwelldefined} until a minimal presentation for $M$ is obtained.  In cases where the generators of $M'$ are significantly smaller than those of $M$, the resulting computation is much faster than directly computing a minimal presentation for $M$.  

Table~\ref{tb:minprescomputation} gives a sample of the improved runtimes that result from using Theorem~\ref{t:minpresbij}.  All runtimes were obtained in the computer algebra system \texttt{GAP} and the \texttt{numericalsgps} package, a standard setting for numerical semigroup computations.  An improved implementation of the function \texttt{MinimalPresentationOfNumericalSemigroup} that utilizes Theorem~\ref{t:minpresbij} is currently in development, and will be available with the next major release of the \texttt{numericalsgps} package.  
\end{remark}

\begin{table}
\begin{tabular}{|l|l|l|l|l|}
$n$ & $M_n$ & \texttt{GAP} Runtime \cite{numericalsgpsgap} & Remark~\ref{r:minprescomputation} \\
\hline
$50$ & $\<50,56,59,70\>$ & 1 ms & 1 ms \\
$200$ & $\<200, 206, 209, 220\>$ & 40 ms & 40 ms \\
$400$ & $\<400, 406, 409, 420\>$ & 210 ms & 210 ms \\
$1000$ & $\<1000, 1006, 1009, 1020\>$ & 3 sec & 210 ms \\
$5000$ & $\<5000, 5006, 5009, 5020\>$ & 18 min & 210 ms \\ 
$10000$ & $\<10000, 10006, 10009, 10020\>$ & 4.2 hr & 210 ms \\ 
\end{tabular}
\medskip
\caption{Runtime comparison for computing a minimal presentation for numerical monoids $M_n$ with $S = \<6,9,20\>$.  All computations performed using \texttt{GAP} and the package \texttt{numericalsgps} \cite{numericalsgpsgap}.}
\label{tb:minprescomputation}
\end{table}

As a corollary of Theorem~\ref{t:minpresbij}, we obtain an improved bound for a result conjectured by Herzog and Srinivasan and later proved by Vu in \cite{vu14}.

\begin{cor}\label{c:bettiperiodic}
The function $n \mapsto |\Betti(M_n)|$ is $r_k$-periodic for $n > r_k^2$.
\end{cor}

\begin{proof}
Given any minimal presentation $\rho$ for $M_n$, the set $\Betti(M_n)$ consists of precisely the elements with factorizations appearing in $\rho$.  Now, Theorem~\ref{t:mesalemma} and Corollary~\ref{c:bettilengthset} imply that each relation in $\rho$ involving factorizations of different lengths produces a distinct Betti element, since the corresponding Betti element $\beta$ must have exactly two connected components in $\nabla\!_\beta$.  Thus, Proposition~\ref{p:mapwelldefined} implies $\Phi_n$ induces a well-defined map $\Betti(M_n) \to \Betti(M_{n+r_k})$, and Theorem~\ref{t:minpresbij} ensures this map is a bijection.  
\end{proof}

\begin{remark}\label{r:syzygies}
The elements of a minimal presentation of a numerical monoid $M$ correspond to binomial generators of the defining toric ideal $I$ of $M$.  In fact, each minimal presentation corresponds to a minimal binomial generating set for $I$.  It is in this setting that Vu approached Corollary~\ref{c:bettiperiodic} in \cite{vu14}, where the Betti elements of $M$ correspond to Betti numbers of $I$ in homological degree 1.  
\end{remark}

\begin{remark}\label{r:vucompare}
Corollary~\ref{c:bettiperiodic}, as well as some results in Sections~\ref{sec:sufficientshift} and~\ref{sec:minpres}, appears in~\cite{vu14} using the language of Remark~\ref{r:syzygies}.  However, our approach has several advantages.  
\begin{enumerate}[(a)]
\item 
Our approach is purely combinatorial; shedding the dependence on commutative algebra makes the results available to a broader mathematical audience, and better isolates the core structural changes (i.e.\ the existence of monotone chains and Theorem~\ref{t:mesalemma}) that occur once $n$ is large enough.  

\item 
Our bound $n > r_k^2$ is lower than each of those previously given, which is crucial for effective use in computation in Remark~\ref{r:minprescomputation}.  Additionally, great care was taken to ease future improvements on our bound, as discussed in Remark~\ref{r:boundorigin}.  

\item 
Several results in this section, such as Corollary~\ref{c:catenaryquasi}, do not follow as directly from statements in~\cite{vu14} as they do from Theorem~\ref{t:minpresbij}.  Indeed, much of the theory developed in Section~\ref{sec:minpres} would have been necessary for a specialized proof of Corollary~\ref{c:catenaryquasi}, and such specialization would have obscured the underlying connection to the other consequences of Theorem~\ref{t:minpresbij} presented here.  

\end{enumerate}
\end{remark}



\begin{remark}\label{r:bettidichotomy}
As a consequence of Theorem~\ref{t:minpresbij} and Corollary~\ref{c:eqlen}, the elements of $\Betti(M_n)$ fall into two distinct categories: those with minimal relations of equal length, and those with minimal relations of different length.  Upon successive applications of $\Phi_n$, those Betti elements in the former category increase linearly with $n$ (with slope given by factorization length, preserved under $\Phi_n$), and those Betti elements in the latter category increase quadratically with $n$.  The plot in Figure~\ref{f:bettidichotomy} exhibits a graphical representation, which makes the distinction more explicit.  
\end{remark}

\begin{figure}
\begin{center}
\includegraphics[width=6.0in]{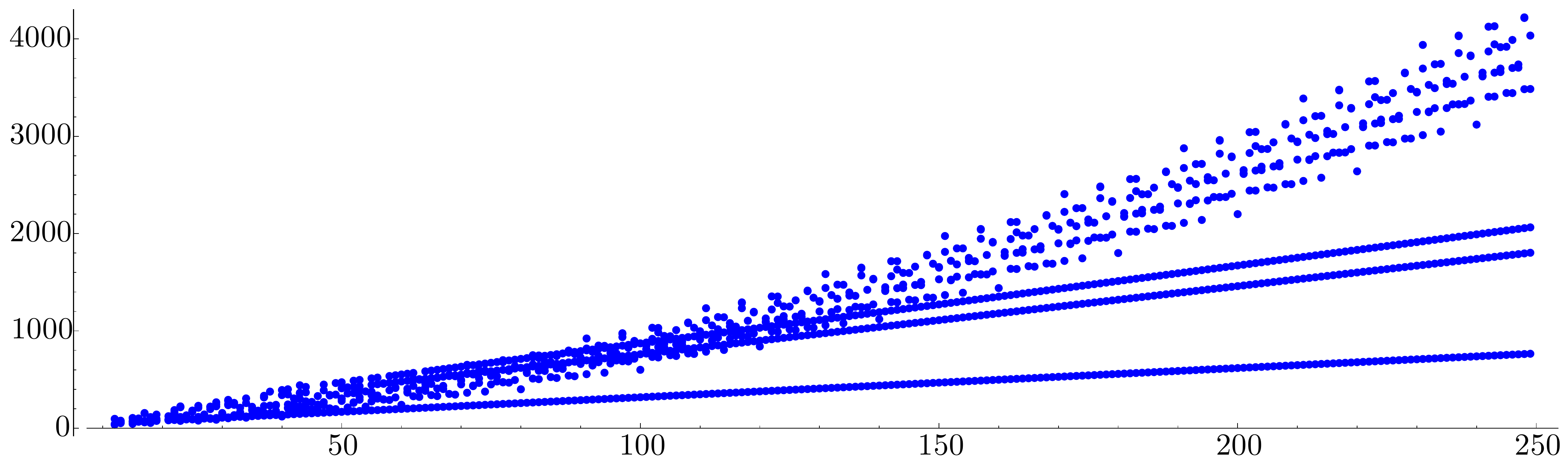}
\end{center}
\caption{A plot depicting the Betti elements of $M_n$ for $S = \<6,9,20\>$ and $n \le 250$.  Each point $(n,b)$ indicates $b \in \Betti(M_n)$.}
\label{f:bettidichotomy}
\end{figure}

Theorem~\ref{t:minpresbij} can also be applied to characterize arithmetic invariants of non-unique factorization for sufficiently large $n$.  We begin with the delta set invariant.  

\begin{defn}\label{d:deltaset}
Fix a numerical monoid $M \subset \NN$, and fix $a \in M$.  Writing 
$$\mathsf L(a) = \{\ell_1 < \cdots < \ell_r\}$$
for the set of distinct factorization lengths of $a$, the \emph{delta set} of $a$ is the set 
$$\Delta(a) = \{\ell_i - \ell_{i-1} : 2 \le i \le r\}$$
of successive differences of factorization lengths of $a$.  Lastly, the \emph{delta set} of $M$ is the union $\Delta(M) = \bigcup_{a \in M} \Delta(a)$ of the delta sets of its elements.  
\end{defn}

As a consequence of Corollary~\ref{c:bettilengthset}, we obtain Corollary~\ref{c:eventualdelta}, which offers an improved bound over \cite[Theorem~2.2]{deltashiftgen}.  

\begin{cor}[{\cite[Theorem~2.2]{deltashiftgen}}]\label{c:eventualdelta}
If $n > r_k^2$, then 
$$\Delta(M_n) = \{d\},$$
where $d = \gcd(r_1, \ldots, r_k)$.  
\end{cor}

\begin{proof}
An elementary number theory argument implies $d = \min\Delta(M_n)$.  Additionally, $\max \Delta(M_n)$ occurs in the delta set of a Betti element of $M_n$ by \cite[Theorem~2.5]{deltabetti}, so $\max \Delta(M_n) = d$ by Corollary~\ref{c:bettilengthset}.  
\end{proof}

Next, we examine the family of catenary degree invariants.  An introduction to the catenary degree is provided in \cite[Section~5]{numericalsurvey}, and an extensive overview of numerous catenary degree variations can be found in \cite{halffactorial}.  

\begin{defn}\label{d:catenarydegree}
Fix a numerical monoid $M = \<m_1, \ldots, m_t\>$ and an element $a \in M$.  For $z, z' \in \mathsf Z(a)$, the \emph{greatest common divisor} of $z$ and $z'$ is given by 
$$\gcd(z,z') = (\min(z_1,z_1'), \ldots, \min(z_t,z_t')) \in \NN^t,$$
and the \emph{distance} between $z$ and $z'$ is given by 
$$d(z,z') = \max(|z - \gcd(z,z')|,|z' - \gcd(z,z')|).$$
For $z, z' \in \mathsf Z(a)$ and $N \ge 1$, an \emph{$N$-chain} from $z$ to $z'$ is a sequence $w_0, \ldots, w_r \in \mathsf Z(a)$ of factorizations of $a$ such that $w_0 = z$, $w_r = z'$, and $d(w_{i-1},w_i) \le N$ for all $i \le r$.  

\begin{enumerate}[(a)]
\item 
The \emph{catenary degree} of $a$, denoted $\mathsf c(a)$, is the smallest $N \in \NN$ such that there exists an $N$-chain between any two factorizations of $a$.  

\item 
The \emph{monotone catenary degree} of $a$, denoted $\mathsf c_\textnormal{mon}(a)$, is the smallest $N \in \NN$ such that there exists a \emph{monotone} $N$-chain (i.e.\ an $N$-chain whose factorization lengths form a monotone sequence) between any two factorizations of $a$.  

\item 
The \emph{equal catenary degree} of $a$, denoted $\mathsf c_\textnormal{eq}(a)$, is the smallest $N \in \NN$ such that there exists an \emph{equal} $N$-chain (i.e.\ an $N$-chain whose factorization lengths are all identical) between any two equal-length factorizations of $a$.  

\end{enumerate}
For each invariant above, define $\mathsf c_*(M) = \sup_{a \in M} \mathsf c_*(a)$.  
\end{defn}

\begin{cor}\label{c:catenaryquasi}
The function $n \mapsto \mathsf c(M_n)$ is eventually quasilinear.  In particular, if~$n > r_k^2$ and $M_n$ is primitive, then 
$$\mathsf c(M_{n+r_k}) = \mathsf c(M_n) + d$$
for $d = \gcd(r_1, \ldots, r_k)$.  
\end{cor}

\begin{proof}
Let $\rho_1, \ldots, \rho_m$ denote all minimal presentations of $M_n$, and let 
$$\mu_i = \max\{|z|, |z'| : (z,z') \in \rho_i\}.$$
By \cite[Theorem~4]{catenarytamefingen}, the catenary degree of $M_n$ equals 
$$\mathsf c(M_n) = \min\{\mu_i : 1 \le i \le m\}.$$
By Theorem~\ref{t:minpresbij}, each $\rho_i$ satisfies $\mu(\Phi_n(\rho_i)) = \mu(\rho_i) + d$, which implies the claim.  
\end{proof}

\begin{cor}\label{c:moncatenaryquasi}
For $n > r_k^2$, we have $\mathsf c_\textnormal{mon}(M_n) = \mathsf c_\textnormal{eq}(M_n) = \mathsf c(M_n)$.  In particular, $\mathsf c_\textnormal{mon}(M_n)$ and $\mathsf c_\textnormal{eq}(M_n)$ are both eventually quasilinear with period $r_k$ as functions of~$n$.  
\end{cor}

\begin{proof}
Definition~\ref{d:catenarydegree} implies $\mathsf c(M_n) \le \mathsf c_\textnormal{eq}(M_n) \le \mathsf c_\textnormal{mon}(M_n)$ for each $n$, so it suffices to prove $\mathsf c_\textnormal{mon}(M_n) = \mathsf c(M_n)$.  Now apply Proposition~\ref{p:monotonechain} and Corollary~\ref{c:catenaryquasi}.  
\end{proof}

\begin{remark}\label{r:catenaryquasi}
Included in Figure~\ref{f:catenaryquasi} is a graphical representation of the quasilinear behavior described in Corollaries~\ref{c:catenaryquasi} and~\ref{c:moncatenaryquasi}.  
\end{remark}

\begin{figure}
\begin{center}
\includegraphics[width=6.0in]{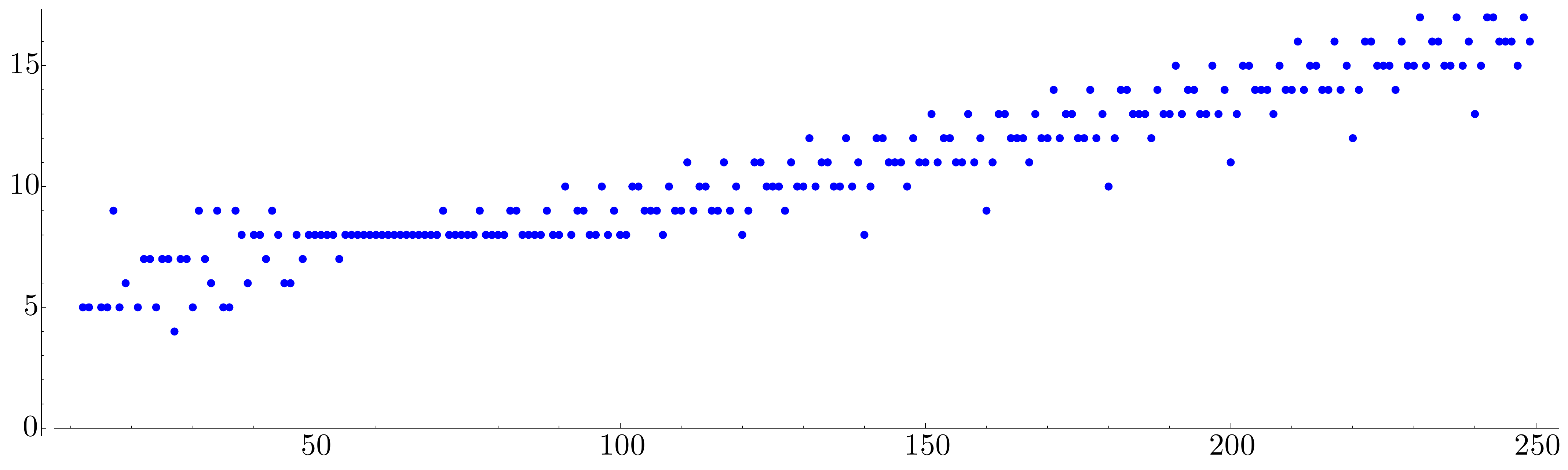}
\end{center}
\caption{A plot depicting $\mathsf c(M_n)$ for $S = \<6,9,20\>$ and $n \le 250$.}
\label{f:catenaryquasi}
\end{figure}

We conclude with Example~\ref{e:tamedegree} demonstrating why a characterization of the eventual behavior of the tame degree (Definition~\ref{d:tamedegree}) does not follow directly from Theorem~\ref{t:minpresbij}.  As such, a solution to Problem~\ref{pr:tamedegree} will likely require a characterization of the primitive elements of $M_n$ for sufficiently large $n$; see~\cite{catenarytamenumerical}.

\begin{defn}\label{d:tamedegree}
Resume notation from Definition~\ref{d:catenarydegree}.  The \emph{tame degree} of $a \in M$, denoted $\mathsf t(a)$, is the smallest $N \in \NN$ such that for each $z \in \mathsf Z(a)$ and each $i \le t$ with $a - m_i \in M$, there exists $z' \in \mathsf Z(a)$ such that $z_i' > 0$ and $d(z,z') \le N$.  
\end{defn}

\begin{example}\label{e:tamedegree}
Unlike the catenary degree, the tame degree of a numerical monoid need not be achieved at a Betti element, even for $n > r_k^2$.  As such, Theorem~\ref{t:minpresbij} does not allow us to immediately characterize the eventual behavior of the tame degree.  Indeed, for $S = \<6,9,20\>$ and $n = 401$, we have $\mathsf c(M_n) = 23$ and $\mathsf t(10869) = 27$.  In contrast, the monotone and equal catenary degrees also need not occur at a Betti element in general (see \cite{halffactorial}), but Corollary~\ref{c:moncatenaryquasi} ensures they do for $n > r_k^2$.  
\end{example}

\begin{prob}\label{pr:tamedegree}
Characterize the tame degree of $M_n$ for $n$ sufficiently large.  
\end{prob}









\section{Acknowledgements}

Much of this work was completed during the Pacific Undergraduate Research Experience in Mathematics (PURE Math), funded by National Science Foundation grants DMS-1035147 and DMS-1045082 and a supplementary grant from the National Security Agency. The authors would like to thank Scott Chapman for giving the initial motivation to start this work and for his helpful comments.



\begin{thebibliography}{20}
\raggedbottom

\bibitem{elastsets}
T. Barron, C. O'Neill, and R. Pelayo,
\emph{On the set of elasticities in numerical monoids},
to appear, Semigroup Forum.  
Available at \textsf{arXiv: math.CO/1409.3425}.

\bibitem{delta}
C. Bowles, S. Chapman, N. Kaplan, D. Reiser, 
\emph{On delta sets of numerical monoids}, 
J. Algebra Appl.  \textbf{5} (2006) 1--24.


\bibitem{catenarytamenumerical}
S.~Chapman, P.~Garc\'ia-S\'anchez, and D.~Llena, 
\emph{The catenary and tame degree of numerical semigroups}, 
Forum Math.\ \textbf{21} (2009) 117-129.

\bibitem{deltabetti}
S.~Chapman, P.~Garc\'ia-S\'anchez, D.~Llena, A.~Malyshev, and D.~Steinberg, 
\emph{On the Delta set and the Betti elements of a BF-monoid}, 
Arab J. Math 1 (2012), 53--61.

\bibitem{catenarytamefingen}
S.~Chapman, P.~Garc\'ia-S\'anchez, D.~Llena, V.~Ponomarenko, J.~Rosales,
\emph{The catenary and tame degree in finitely generated commutative cancellative monoids},
Manuscripta Math. 120 (2006), no. 3, 253--264. 


\bibitem{deltashiftgen}
S.~Chapman, N.~Kaplan, T.~Lemburg, A.~Niles, and C.~Zlogar, 
\emph{Shifts of generators and delta sets of numerical monoids}, 
Internat.~J.~Algebra Comput.~24 (2014), no.~5, 655--669.

\bibitem{clo}
D.~Cox, J.~Little, and D.~O'Shea,
\emph{Ideals, varieties, and algorithms: An introduction to computational algebraic geometry and commutative algebra},
4th edition, Undergraduate Texts in Mathematics, Springer, Cham, 2015. 

\bibitem{numericalsgpsgap}
M.~Delgado, P.~Garc\'ia-S\'anchez, J.~Morais, 
\emph{NumericalSgps, A package for numerical semigroups}, 
Version 0.980 dev (2013), (GAP package),
\url{http://www.fc.up.pt/cmup/mdelgado/numericalsgps/}.

\bibitem{algmarkov}
P.~Diaconis and B.~Sturmfels, 
\emph{Algebraic algorithms for sampling from conditional distributions},
Ann.\ Statist.\ 26 (1998), no.~1, 363--397.

\bibitem{compoverview} 
P.~Garc\'ia-S\'anchez,
\emph{An overview of the computational aspects of nonunique factorization invariants}, 
preprint.  Available at \textsf{arXiv: math.AC/1504.07424}

\bibitem{halffactorial}
P.~Garc\'ia-S\'anchez, I.~Ojeda, R.~S\'anchez, and A.~Navarro, 
\emph{Factorization invariants in half-factorial affine semigroups},
Internat.\ J.\ Algebra Comput.\ 23 (2013), no.~1, 111--122. 


\bibitem{numericalsurvey}
C. O'Neill and R. Pelayo, 
\emph{Factorization invariants in numerical monoids},
to appear, Contemp.\ Math.  Available at \textsf{arXiv: math.AC/1508.00128}.



\bibitem{fingenmon}
J.C. Rosales and P.A. Garc\'ia-S\'anchez, 
\emph{Finitely generated commutative monoids}, 
Nova Science Publishers, Inc., Commack, NY, 1999. xiv+185 pp. ISBN: 1-56072-670-9.  

\bibitem{numerical}
J.~Rosales and P.~Garc\'ia-S\'anchez, 
\emph{Numerical semigroups}, 
Developments in Mathematics, Vol. 20, Springer-Verlag, New York, 2009.

\bibitem{vu14}
T.~Vu, 
\emph{Periodicity of Betti numbers of monomial curves},
Journal of Algebra 418 (2014) 66--90.

\end{thebibliography}
\end{document}